\documentclass{article}
\usepackage[english]{babel}
\usepackage[utf8]{inputenc}
\usepackage{amsmath}
\usepackage{graphicx}
\usepackage{amsfonts}
\usepackage{enumitem}
\usepackage{amssymb}
\usepackage[colorinlistoftodos]{todonotes}
\usepackage{geometry}
\usepackage{amsthm}

\newcommand{\norm}[1]{\left\lVert#1\right\rVert}

\newcommand{\ang}[1]{\left\langle #1 \right\rangle}
\newcommand{\floor}[1]{\left\lfloor #1 \right\rfloor}

\newcommand{\paren}[1]{\left( #1 \right)}

\newcommand{\set}[1]{\left\{ #1 \right\}}
\newcommand{\setcond}[2]{\left\{ #1 \;\middle\vert\; #2 \right\}}

\newcommand{\RR}{\mathbb{R}}

\newcommand{\ZZ}{\mathbb{Z}}
\newcommand{\cP}{\mathcal{P}}
\newcommand{\cQ}{\mathcal{Q}}

\newtheorem{thm}{Theorem}[section]
\newtheorem{lem}[thm]{Lemma}
\newtheorem{cor}[thm]{Corollary}
\newtheorem*{rmk}{Remark}

\theoremstyle{definition}
\newtheorem{defn}[thm]{Definition}
\newtheorem{example}[thm]{Example}

\DeclareMathOperator{\conv}{Conv}
\DeclareMathOperator{\Int}{Int}

\title{Fixing a hole}
\author{David Conlon\thanks{Department of Mathematics, Caltech, Pasadena, CA 91125, USA. Email: {\tt dconlon@caltech.edu}. Research supported by NSF Award DMS-2054452.} \and Jeck Lim\thanks{Department of Mathematics, Caltech, Pasadena, CA 91125, USA. Email: {\tt jlim@caltech.edu}. Research partially supported by the NUS Overseas Graduate Scholarship.}}
\date{}

\begin{document}
\maketitle

\begin{abstract}
We show that any finite $S \subset \mathbb{R}^d$ in general position has arbitrarily large supersets $T \supseteq S$ in general position with the property that $T$ contains no empty convex polygon, or hole, with $C_d$ points, where $C_d$ is an integer that depends only on the dimension $d$. This generalises results of Horton and Valtr which treat the case $S = \emptyset$. The key step in our proof, which may be of independent interest, is to show that there are arbitrarily small perturbations of the set of lattice points $[n]^d$ with no large holes.
\end{abstract}

\section{Introduction}

A set $S \subset \mathbb{R}^d$ is said to be in \emph{general position} if, for any $k < d$, there are at most $k+1$ points in any $k$-dimensional subspace, while it is in \emph{convex position} if the points of $S$ form the vertices of a convex polytope. A classic result of Erd\H{o}s and Szekeres~\cite{ESz35} then states that for any $\ell$ there exists $n$ such that any set of $n$ points in general position in the plane contains $\ell$ points in convex position. The analogous statement in higher dimensions also follows as a simple corollary.

Our concern here will be with a variant introduced by Erd\H{o}s~\cite{E78, E81}. Given a set $S \subset \mathbb{R}^d$, we say that points $s_1, \dots, s_\ell \in S$ form an \emph{$\ell$-hole} if they are in convex position and no point of $S$ is contained in the interior of the convex polytope formed by $s_1, \dots, s_\ell$. Erd\H{o}s' question was whether, for each $\ell$, there exists $n$ such that any set of $n$ points in general position in the plane contains an $\ell$-hole. That such an $n$ exists for $\ell = 5$ was proved by Harborth~\cite{H78} in 1978, though it took almost thirty years more for the $\ell = 6$ case to be solved in the affirmative by Nicol\'as~\cite{N07} and, independently, Gerken~\cite{G08}. At least in the plane, this is where the story ends, since there is a remarkable construction, due to Horton~\cite{H83}, of arbitrarily large point sets in general position with no $7$-hole.

In higher dimensions, Horton-type sets were first constructed by Valtr~\cite{V92}, who found arbitrarily large $T \subset \mathbb{R}^d$ in general position containing no $B_d$-hole, where $B_d = d^{d + o(d)}$. Very recently, a more efficient construction was given by Bukh, Chao and Holzman~\cite{BCH21} (see also~\cite{BC21}), who showed that one may take $B_d = 4^{d +o(d)}$. Surprisingly, the best lower bound remains that of Valtr~\cite{V92}, which says that sufficiently large point sets in general position in $\mathbb{R}^d$ contain $(2d+1)$-holes.

Suppose, however, that one starts with a large hole or, say, a large random point set, which are known~\cite{BGS13} to contain many large holes. Is it then possible to add points to the set to obtain a set with no large holes? It is this natural question that we address here, our main result saying that any finite $S \subset \mathbb{R}^d$ may be filled out to form a set without large holes. Note that here and throughout, we will say that a set $S \subset \RR^d$, which need not be in general position, is {\it $\ell$-hole-free} if, for any set of $\ell$ points $s_1,\ldots,s_\ell\in S$, there is a point $s\in S$ in the interior of the convex hull of $\{s_1, s_2, \dots, s_\ell\}$.

\begin{thm} \label{thm:main}
For any integer $d \geq 2$, there exists an integer $C_d=d^{O(d^3)}$ such that if $S \subset \mathbb{R}^d$ is a finite set in general position, then there are arbitrarily large supersets $T \supseteq S$ in general position with the property that $T$ is $C_d$-hole-free. In particular, when $d = 2$, one may take $C_2 = 9$.
\end{thm}

We suspect that this theorem may remain true in two dimensions with $C_2 = 7$. Our methods do not suffice to show this, so we leave it instead as a tantalising open problem.

In practice, Theorem~\ref{thm:main} will follow from another theorem of independent interest, saying that there is a set with no large holes which approximates the set of lattice points $[n]^d$. For $d = 2$, such a theorem is already implicit in work of Valtr~\cite{V92-2}. The main technical result of this paper is the analogous result for higher dimensions.

\begin{thm} \label{thm:lattice}
For any integers $n \geq 1$ and $d \geq 2$ and any $\varepsilon > 0$, there exists an integer 
$C_d'=d^{O(d^3)}$ and a set of points $\cP=\setcond{P_{\vec{x}}}{\vec{x}\in [n]^d}\subset\RR^d$ that is $C_d'$-hole-free and satisfies $\norm{P_{\vec{x}}-\vec{x}}<\varepsilon$ for all $\vec{x}\in [n]^d$. In particular, when $d = 2$, one may take $C_2' = 7$.
\end{thm}

Given finite sets $A,B\subset \RR^d$ and $\varepsilon>0$, we call a bijection $f:A\to B$ an \emph{$\varepsilon$-perturbation} if  $\norm{x-f(x)}<\varepsilon$ for all $x\in A$. With this notation, Theorem~\ref{thm:lattice} can be reformulated as saying that for any $\varepsilon > 0$ there exists an $\varepsilon$-perturbation of $[n]^d$ which is $C'_d$-hole-free. For brevity, we will usually discuss Theorem~\ref{thm:lattice} in these terms.

Notice that if a finite set $T\subset\RR^d$ is $\ell$-hole-free (where $T$ may not be in general position), then any sufficiently small perturbation of $T$ will still be $\ell$-hole-free. Indeed, this is why we define $\ell$-hole-free as we do rather than simply saying it is $\ell$-hole-free if it contains no $\ell$-hole. Moreover, for any $T$, there are arbitrarily small perturbations of $T$ which put $T$ in general position. Hence, to prove Theorem~\ref{thm:main}, we only need to find arbitrarily large $\ell$-hole-free supersets $T$ of $S$ without worrying about whether or not they are in general position. With this observation, we can quickly show how  Theorem~\ref{thm:main} follows  from Theorem~\ref{thm:lattice}.

\begin{proof}[Proof of Theorem~\ref{thm:main} given Theorem~\ref{thm:lattice}]
Let $n$ be sufficiently large in terms of $S$ and let $\cP$ be as given by Theorem~\ref{thm:lattice} with this $n$ and $\varepsilon = 1/10$. Set $C_d=C_d'+d$. Since $S$ is in general position and $n$ (and, hence, $|\cP|$) is sufficiently large in terms of $S$, we can scale and translate $\cP$ in such a way that any $d+1$ point subset of $S$ has a point of this homothetic copy $\cQ$ of $\cP$ in the interior of its convex hull. We then set $T=S\cup \cQ$. Suppose now that $A\subset T$ is such that $\Int\conv A$ does not contain a point of $T$. Then $A$ cannot contain $d+1$ points of $S$, since otherwise the interior of its convex hull will contain a point of $\cQ$. $A$ also cannot contain $C_d'$ points of $\cQ$, since $\cQ$ is $C_d'$-hole-free. Thus, $|A|\leq C_d'+d-1=C_d-1$ and $T$ is $C_d$-hole-free. If $T$ is not in general position, then, following the paragraph above, we can move it to general position with a sufficiently small perturbation while preserving the fact that it is $C_d$-hole-free.
\end{proof}

Though it already appears in similar terms in the work of Valtr~\cite{V92-2}, we begin by taking a close look at the planar case of Theorem~\ref{thm:lattice}, since it will inform our arguments in higher dimensions.

\section{The planar case}

The following definition will be important throughout the paper.

\begin{defn}
A \emph{levelled set} $L$ is a subset of $\RR^d$ together with a surjective affine map $\phi_L:\RR^d\to \RR$ such that $\phi_L(L)\subset\ZZ$. We say that $\phi_L$ is the \emph{level map} of $L$.
\end{defn}

Note that a subset of a levelled set is also a levelled set with the same level map. In practice, we will often be interested in a particular type of subset.

\begin{defn}
Given a levelled set $L\subset\RR^d$ and integers $a$ and $p$ with $p \geq 1$, define the set
$$L_{a,p}=\setcond{x\in L}{\phi_L(x)\equiv a \; (\bmod{\; p})}.$$
We can make $L_{a,p}$ into a levelled set with the level map $\phi_{L_{a,p}}(x)=(\phi_L(x)-a)/p$. Note that this level map is not the same as when $L_{a,p}$ is simply viewed as a levelled subset of $L$.
\end{defn}

We also fix some notation that we will use throughout, writing $\vec{e}_1,\vec{e}_2,\ldots,\vec{e}_d$ for the standard basis of $\RR^d$ and $\pi_i:\RR^d\to\RR$ for the projection onto the $i$-th coordinate for any $i = 1, 2, \dots, d$.

Now we make some definitions which are more specific to two dimensions, only generalising to higher dimensions later.

\begin{defn} \label{def:2ha}
For finite sets $A,B\subset\RR^2$, we say that \emph{$A$ lies high above $B$} and that \emph{$B$ lies deep below $A$} if:
\begin{enumerate}
\item for any pair of points $p,q\in A$ with distinct $x$-coordinates, the entire set $B$ lies below the line joining $p$ and $q$;
\item for any pair of points $p,q\in B$ with distinct $x$-coordinates, the entire set $A$ lies above the line joining $p$ and $q$.
\end{enumerate}
\end{defn} 

It is not hard to see that given two finite sets $A$ and $B$ we can always shift $A$ upwards so that $A$ lies high above $B$. More precisely, there is some $M$ such that, for $m>M$, the set $A+m\vec{e}_2$ lies high above $B$. 

Our definition of a Horton set is similar to that in \cite{V92} and \cite{V92-2}, but rephrased in terms of levelled sets.

\begin{defn}
\label{def:2horton}
Let $H\subset\RR^2$ be a levelled set with level map of the form $\phi_H=a\pi_1+b$ for some $a,b\in\RR$ with $a\neq 0$ and suppose that $\phi_H$ is injective on $H$ with $\phi_H(H)$ a consecutive set of integers. We say that $H$ is \emph{Horton} if it is Horton according to a finite number of applications of the following rules:
\begin{enumerate}
\item The empty set or any singleton levelled set is Horton.
\item If $H_{0,2}$ and $H_{1,2}$ are Horton and $H_{0,2}$ lies deep below or high above $H_{1,2}$, then $H$ is Horton.
\end{enumerate}
\end{defn}

\begin{example}
\label{ex:2horton}
Following \cite{H83, V92, V92-2}, we can construct Horton sets as follows. For any positive integer $N$, write it in binary as $N=\sum_{k\geq 0} a_k 2^k$, where $a_k\in\{0,1\}$. Denote by $(N)_\varepsilon$ the real number 
$$(N)_\varepsilon=\sum_{k\geq 0} a_k \varepsilon^{k+1},$$
noting that $0<(N)_\varepsilon<2\varepsilon$ for $\varepsilon<1/2$. Consider the set $S=\{P_x\mid x\in [n]\}$, where $P_x=(x,(x)_\varepsilon)$. Observe that, for $\varepsilon$ sufficiently small, $S_{0,2}$ lies deep below $S_{1,2}$. By recursively applying this observation to the sets $S_{0,2}$ and $S_{1,2}$, we can easily see that $S$ is Horton for $\varepsilon$ sufficiently small.
\end{example}

To say something about the properties of Horton sets, we require some further definitions.

\begin{defn}
A sequence of points $p_1,\ldots,p_r\in\RR^2$ with $\pi_1(p_1)<\cdots<\pi_1(p_r)$ is said to be
\begin{enumerate}
    \item \emph{convex} if, for all $i, j, k$ with $1 \leq i < j < k \leq r$, the point $p_j$ lies below or on the straight line $p_ip_k$,
    \item \emph{concave} if, for all $i, j, k$ with $1 \leq i < j < k \leq r$, the point $p_j$ lies above or on the straight line $p_ip_k$.
\end{enumerate}
\end{defn}

\begin{defn}
A convex sequence of $r$ points $p_1, p_2,\ldots,p_r \in\RR^2$ with $\pi_1(p_1)<\cdots<\pi_1(p_r)$ is \emph{upper closed} by a point $p$ if $\pi_1(p_1)< \pi_1(p)< \pi_1(p_r)$ and the point $p$ lies above the polygonal line $p_1p_2\ldots p_r$. Similarly, a concave sequence of $r$ points $p_1,p_2,\ldots,p_r \in\RR^2$ with $\pi_1(p_1)<\cdots<\pi_1(p_r)$ is \emph{lower closed }by a point $p$ if $\pi_1(p_1)< \pi_1(p)< \pi_1(p_r)$ and the point $p$ lies below the polygonal line $p_1p_2\ldots p_r$.
\end{defn}

\begin{defn}
Let $A\subset\RR^2$ be a finite set of points for which $\pi_1$ is injective. $A$ is said to be
\begin{enumerate}
    \item \emph{upper $r$-closed} if every convex sequence of $r$ points from $A$ is upper closed by some point of $A$,
    \item \emph{lower $r$-closed} if every concave sequence of $r$ points from $A$ is lower closed by some point of $A$,
    \item \emph{$r$-closed} if it is both upper and lower $r$-closed. 
\end{enumerate}
\end{defn}

Given these definitions, we can record the following results from \cite{V92-2}.

\begin{lem} \label{lem:rshole}
Let $A, B \subset\RR^2$ be finite sets of points which are each $(r+s-1)$-hole-free. Suppose that $\pi_1$ is injective on $A$ and $B$, that $A$ is upper $r$-closed, that $B$ is lower $s$-closed and that $A$ lies deep below $B$. Then the set $A\cup B$ is also $(r + s - 1)$-hole-free. 
\end{lem}

\begin{lem} \label{lem:subhorton2}
If $H$ is Horton, then, for any integers $a$ and $p$ with $p \geq 1$, the set $H_{a,p}$ is also Horton.
\end{lem}

\begin{lem} \label{lem:7hole}
Any Horton set is 4-closed and 7-hole-free.
\end{lem}

We also note one further result about Horton sets.

\begin{lem} \label{lem:double7hole}
Let $H_1$ and $H_2$ be Horton and suppose $H_1$ lies deep below $H_2$. Then $H_1\cup H_2$ is 7-hole-free.
\end{lem}

\begin{proof}
By Lemma~\ref{lem:7hole}, $H_1$ and $H_2$ are 4-closed and 7-hole-free. Thus, by Lemma~\ref{lem:rshole}, $H_1\cup H_2$ is 7-hole-free.
\end{proof}

The following lemma characterises those finite sets of lattice points (that is, those finite subsets of $\ZZ^2$) which contain no lattice point in the interior of their convex hull. Once again, this result may be found in~\cite{V92-2}. However, since finding an appropriate generalisation of this result is one of the key steps in higher dimensions, we give the proof in full.

\begin{lem}
\label{lem:parallel}
Let $S$ be a finite set of lattice points of size at least 7. Then either the interior of the convex hull of $S$ contains a lattice point or $S$ is covered by 2 parallel lines.
\end{lem}

\begin{proof}
Suppose the interior of the convex hull does not contain any lattice point. Then $S$ forms a non-strictly convex polygon. 
Colour the points of $S$ with 4 colours based on the parities of their $x$ and $y$-coordinates. Then, by the pigeonhole principle, there are 2 points of the same colour, so their midpoint is a lattice point. Let $l_0$ be the line through these 2 points. By applying an appropriate invertible affine transformation that takes the set of lattice points $\ZZ^2$ to itself, we can assume that $l_0$ is the line $y=0$ and there are two points of $S$ on $l_0$, namely, $P_0=(0,0)$ and $P_1=(x_1,0)$, with $x_1>0$ even. Let $l_m$ be the line $y=m$. Note that there cannot be points of $S$ in both the upper and lower half plane, since otherwise the midpoint of $P_0$ and $P_1$ lies in the interior of the convex hull. So we may assume that all the points of $S$ lie in the upper half plane. There cannot be a point on $l_m$ for $m\geq 3$, since if there is some point $Q\in l_m$, then the triangle $P_0P_1Q$ intersects the line $l_1$ in a segment of length greater than $1$ and so contains an interior point from $l_1$. 
There also cannot be 2 points on $l_2$, since otherwise the trapezium formed by those 2 points together with $P_0$ and $P_1$ intersects $l_1$ in a line segment of length greater than 1 and there will again be an interior point on $l_1$. Finally, if there is a point $Q\in l_2$, then we must have $x_1=2$, since otherwise the triangle $P_0P_1Q$ again intersects the line $l_1$ in a line segment of length greater than 1. But then there can be at most 6 points in $S$: 3 on $l_0$, 2 on $l_1$ and 1 on $l_2$. Thus, there are no points on $l_2$ and all points of $S$ lie on the pair of parallel lines $l_0$ and $l_1$.
\end{proof}

We are now almost ready to prove Theorem~\ref{thm:lattice} in the $d = 2$ case, though we need one more key definition.

\begin{defn}
Given finite sets $A,B\subset \RR^d$, 
a bijection $f:A\to B$ 
is \emph{negligible} if, whenever $S\subset A$ and $x\in A$ with $x\in \Int\conv S$, then $f(x)\in \Int\conv f(S)$. 
\end{defn}

It is not hard to see that for any finite set $A\subset \RR^d$ and any sufficiently small $\varepsilon > 0$ 
every $\varepsilon$-perturbation $f:A\to B$ is negligible. We will use this fact repeatedly in both the proof below and its generalisation to higher dimensions.

\begin{proof}[Proof of Theorem~\ref{thm:lattice} for $d = 2$]
To construct the required set $\cP$, we start with the lattice square $\cP^{(0)}=[n]^2$. Then we shift each column vertically by a small amount to get the set $\cP^{(1)}$, shifting them in such a way that each row of points is Horton, which will also imply that any non-vertical line of lattice points corresponds to a subset of $\cP^{(1)}$ that is Horton. Finally, we shift each row horizontally by an even smaller amount to get $\cP$, so that each column of points is also Horton, though now with respect to projection onto the $y$-axis. In other words, $\cP$ will be a Minkowski sum of two Horton sets, one of them resembling $[n]$ along the $x$-axis, the other along the $y$-axis.

Let $\cP^{(0)}=\setcond{P^{(0)}_{x,y}}{x,y\in [n]}$, where $P^{(0)}_{x,y}=(x,y)$. Define the points $P^{(1)}_{x,y}=P^{(0)}_{x,y}+(x)_\varepsilon\vec{e}_2$ and the set $\cP^{(1)}=\setcond{P^{(1)}_{x,y}}{x,y\in [n]}$. Note that the natural bijection $\cP^{(0)}\to \cP^{(1)}$ is a $2\varepsilon$-perturbation. From Example~\ref{ex:2horton}, we see that for fixed $y$ and small enough $\varepsilon$, the set
$$H=\setcond{P^{(1)}_{x,y}\in \cP^{(1)}}{x\in [n]}$$
is Horton and, hence, so is $H_{a,p}$ for any integers $a$ and $p$ with $p \geq 1$, by Lemma~\ref{lem:subhorton2}. For any non-vertical line $L$, the set
$$\setcond{P^{(1)}_{x,y}\in \cP^{(1)}}{(x,y)\in L\cap [n]^2}$$
is just an affine transformation of $H_{a,p}$ for some $a$ and $p$ 
of the form $(x,y)\mapsto (x,y+cx+d)$ 
and so is also Horton. 

Now, for some $\delta>0$, define the points $P_{x,y}=P^{(1)}_{x,y}+(y)_\delta\vec{e}_1$ and the set $\cP=\{P_{x,y}\mid x,y\in [n]\}$. Here $\delta$ is chosen small enough that the natural bijection $\cP^{(1)}\to \cP$ is negligible and, for any fixed $x$, the set 
$$\Phi\paren{\setcond{P_{x,y}\in \cP}{y\in [n]}}$$
is Horton, where $\Phi:\RR^2\to \RR^2$ is the map that swaps both coordinates. Note that $[n]^2\to \cP$ is a $(2\varepsilon+2\delta)$-perturbation.
Since we can take $\varepsilon$ and $\delta$ arbitrarily small, we may also assume that this map is negligible.
It remains to show that $\cP$ is 7-hole-free.

For any $T\subset\RR^2$, let 
$$\cP_T = \setcond{P_{x,y}\in \cP}{(x,y)\in T\cap [n]^2}.$$
We can similarly define the set $\cP^{(1)}_T$. Now let $S\subset [n]^2$ be of size 7. We wish to show that $\Int\conv \cP_S$ contains a point of $\cP$. By Lemma~\ref{lem:parallel}, either $S$ contains a lattice point in the interior of its convex hull (in which case $\Int\conv \cP_S$ contains a point of $\cP$, since $[n]^2\to \cP$ is negligible) 
or $S$ is contained in 2 parallel lines. Suppose, therefore, that $S$ is contained in 2 parallel lines $L_1$ and $L_2$. If these lines are vertical, then $\Phi(\cP_{L_1})$ and $\Phi(\cP_{L_2})$ are Horton and, for $\delta$ sufficiently small, one of them lies deep below the other, so $\cP_{L_1\cup L_2}$ is 7-hole-free by Lemma~\ref{lem:double7hole}. If instead $L_1$ and $L_2$ are not vertical, then $\cP^{(1)}_{L_1}$ and $\cP^{(1)}_{L_2}$ are Horton and, similarly, $\cP^{(1)}_{L_1\cup L_2}$ is 7-hole-free for $\varepsilon$ sufficiently small. But then, since $\cP^{(1)}\to \cP$ is negligible, $\cP_{L_1\cup L_2}$ is also 7-hole-free. Thus, $\Int\conv \cP_S$ will always contain a point of $\cP$.
\end{proof}

\section{Higher dimensions}

\subsection{Lattice subsets with few points on a hyperplane}

As we have seen in the plane, a set of lattice points does not necessarily have a lattice point in the interior of its convex hull, since they could all lie on a hyperplane. However, if such a set has many points on a hyperplane, then we can instead work within that sublattice of lower dimension. We may therefore assume that no hyperplane contains too many points. We will show below (Lemma~\ref{lem:basic}) that in this case the interior of the set's convex hull contains not just one, but many lattice points.
Unfortunately,  
having a lattice point which is in the interior of the convex hull with respect to the affine space spanned by the set does not guarantee a point in the interior of the convex hull with respect to the whole space $\RR^d$. This problem does not really matter in two dimensions, since the only non-trivial lattice of lower dimension is $\ZZ$ and we already know how to construct Horton sets that look like $\ZZ$ in two (or more) dimensions. In higher dimensions, we solve the problem by showing (Lemma~\ref{lem:well-spread}) that a set of lattice points with few points on any hyperplane must contain ``enough'' lattice points in the interior of its convex hull that when embedded in a higher-dimensional lattice and carefully perturbed, it 
has a point which is in the interior of the convex hull with respect to the larger space.

We will work throughout this section with two notions of cubes, defined as follows.

\begin{defn}
Let $d$ and $r$ be positive integers. A \emph{cube of length $r$} in $\ZZ^d$ is a set of the form
$$\{a+i_1v_1+i_2v_2+\cdots+i_dv_d\mid 0\leq i_1,\ldots,i_d\leq r\}$$
for some $a,v_1,\ldots,v_d\in \ZZ^d$ with $\{v_1,\ldots,v_d\}$ an $\RR$-basis of $\RR^d$. We say that $\{v_1,\ldots,v_d\}$ is the \emph{basis} of the cube and $a$ is the \emph{origin} of the cube. If $\{v_1,\ldots,v_d\}$ is in fact a $\ZZ$-basis of $\ZZ^d$, then we say that $S$ is a \emph{basic cube of length $r$}. 
\end{defn}

The following lemma shows that cubes contain basic cubes in their convex hulls.

\begin{lem}
A cube of length $r$ in $\ZZ^d$ contains a basic cube of length $\lfloor r/d\rfloor$ in its convex hull.
\end{lem}

\begin{proof}
Without loss of generality, we may assume that the origin of the cube is 0. Suppose that the basis of the cube is $v_1,\ldots,v_d$ and let $L_k=\ang{v_1,\ldots,v_k}_{\RR}\cap \ZZ^d$ be the $k$-dimensional lattice spanned by the first $k$ elements of the basis, so we have a flag of lattices $L_1\subset L_2\subset\dots\subset L_d=\ZZ^d$. 

For $k=1,\ldots,d$, we inductively pick vectors $u_k\in L_k$ such that $\ang{u_1,\ldots,u_k}_{\ZZ}=L_k$ as follows. To begin, note that there are two possible choices for $u_1$. Pick the one for which $u_1=c_{11}v_1$ for some $c_{11}>0$. Then we also have $c_{11}\leq 1$. For $k>1$, first pick any valid $u_k$ such that $\ang{u_1,\ldots,u_{k-1},u_k}_{\ZZ}=L_k$ and write it as a linear combination $u_k=c_{k1}v_1+\dots+c_{kk}v_k$ with $c_{kk}\neq 0$. If $c_{kk}<0$, replace $u_k$ by $-u_k$, so we may assume that $c_{kk}>0$. By adding integer multiples of $v_1,\ldots,v_{k-1}$ to $u_k$, we may also assume that $0\leq c_{ki}<1$ for $i=1,\ldots,k-1$. Note that $[u_k]$ is a generator of the quotient $L_k/L_{k-1}\cong \ZZ$, so $[v_k]=n[u_k]$ for some integer $n$. Thus, $c_{kk}=1/n\leq 1$.

In the end, we obtain a $\ZZ$-basis $\{u_1,\ldots,u_d\}$ of $\ZZ^d$ such that $u_k=c_{k1}v_1+\dots+c_{kd}v_d$ and $0\leq c_{ki}\leq 1$ for all $1 \leq k, i \leq d$. Hence, the basic cube of length $\floor{r/d}$ with basis $\{u_1,\ldots,u_d\}$ and origin 0 lies in the convex hull of the cube of length $r$ with basis $\{v_1,\ldots,v_d\}$.
\end{proof}

Our next lemma, already referenced in the discussion above, may be seen as a higher-dimensional analogue of Lemma~\ref{lem:parallel}. It is a Ramsey-type statement showing that, for any sufficiently large set of lattice points in $\RR^d$, either many of them lie on a common hyperplane or the interior of the set's convex hull contains a basic cube. 

\begin{lem}
\label{lem:basic}
Let $d$, $r$ and $m$ be positive integers with $m > d \geq 2$. Then, for any integer $N \geq  m(r+1)^d d^{2d}$ and any set of lattice points $S\subset \ZZ^d$ of size $N$, either there is a hyperplane containing at least $m$ points of $S$ or the interior of the convex hull of $S$ contains a basic cube of length $r$. 
\end{lem}

\begin{proof}
Let $p= r d^2 + 2d$, so that $N \geq mp^d$. Assume that no hyperplane contains $m$ points of $S$. Look at the coordinates of each point of $S$ modulo $p$. Since there are only $p^d$ different choices of congruence class over all $d$ coordinates, the pigeonhole principle implies that there is a subset $T\subset S$ of size at least $N/p^d$ such that, for any two points $u_1,u_2\in T$, the difference $u_1-u_2$ has all coordinates divisible by $p$. Since $N/p^d\geq m$, not all points of $T$ lie on a hyperplane. In other words, $d+1$ of them lie in general position and so are of the form $a,a+pv_1,a+pv_2,\ldots,a+pv_d$ for some $a,v_1,\ldots,v_d\in \ZZ^d$ such that $\{v_1,\ldots,v_d\}$ is an $\RR$-basis. $\Int\conv T$ thus contains the points $a+i_1v_1+\dots+i_dv_d$ for each $1\leq i_1,\ldots,i_d<p/d$, so we have a cube of length at least $p/d-2$. By the previous lemma, we then obtain a basic cube of length $\lfloor\frac{p/d-2}{d}\rfloor\geq r$ in the interior of $\conv T$.
\end{proof}

As mentioned before, simply knowing that a lattice subset contains a single lattice point in the interior of its convex hull is not enough and we instead require it to contain ``enough'' lattice points. Here we make precise what it means to be ``enough''.

\begin{defn}
Let $N$ and $r$ be positive integers and $L\subset \RR^d$ a levelled set. We say that $H\subset L$ is \emph{$(N,r)$-spread in $L$} if, for any $S\subset H$ with at least $N$ points and any integer $a$, $\Int\conv S$ contains a point of $L_{a,r}$. We say that $H\subset L$ is \emph{$(N,r)$-well-spread in $L$} if, for any integers $a$ and $p$ with $p\geq 1$, the set $H_{a,p}$ is $(N,r)$-spread in $L_{a,p}$.
\end{defn}

To get some feel for this definition, we note that saying that $L$ is $(N,1)$-spread in itself is the same as saying that $L$ is $N$-hole-free. Moreover, if the level map is just projection onto the $x$-axis, then $H$ is $(N,r)$-spread in $L$ if any subset of $H$ of size $N$ contains points of $L$ in the interior of its convex hull with any $x$-coordinate modulo $r$. 

Our next lemma is the promised result showing that any sufficiently large set of lattice points in $\RR^d$ either has many points on a hyperplane or ``enough'' points in  the interior of its convex hull. For us, this latter statement will mean that the point set is well-spread in the lattice.
Note that any lattice $L\subset\RR^d$ with $L\cong \ZZ^d$ can be viewed as a levelled set whose level map is projection onto the span of one of the $\ZZ$-basis elements of $L$.

\begin{lem}
\label{lem:well-spread}
Let $d$, $r$ and $m$ be positive integers with $m > d \geq 2$. Then, for any integer $N \geq m(r+1)^d d^{2d}$ and any $d$-dimensional lattice $L\subset\RR^d$, any $H\subset L$ either has at least $m$ points on a hyperplane or is $(N,r)$-well-spread in $L$. 
\end{lem}

\begin{proof}
Suppose $H\subset L$ has no $m$ points on any hyperplane. Then, for any $a$ and $p$ with $p\geq 1$, the set $H_{a,p}$ also has no $m$ points on any hyperplane. Hence, without loss of generality, it suffices to show that $H$ is $(N,r)$-spread in $L$. 

By applying an invertible affine transformation, we may assume that $L$ is $\ZZ^d$ and its level map is simply the projection $\pi_1$ onto the first coordinate. 
Without loss of generality, we may also assume that $H\subset L=\ZZ^d$ has size $N$. We wish to show that $\Int\conv H$ contains a point of $L_{a,r}$ for any integer $a$. By Lemma~\ref{lem:basic}, $\Int\conv H$ contains a basic cube of length $r$ with basis $v_1,\ldots,v_d$ and origin $v$. Let $x'\in\ZZ^d$ be any lattice point with $\pi_1(x')\equiv a\pmod{r}$. We can write $x'=v+c_1'v_1+\dots+c_d'v_d$ for some integers $c_1',\ldots,c_d'$. Set $x=v+c_1v_1+\dots+c_dv_d$, where $c_i\equiv c_i'\pmod{r}$ and $0\leq c_i<r$ for each $i$. Then $\pi_1(x)\equiv a\pmod{r}$ and $x$ is in the basic cube, which is itself contained in $\Int\conv H$.
\end{proof}

\subsection{Horton sets}

In this section, building on Valtr's work~\cite{V92}, we define Horton sets in higher dimensions and note some of their properties. We first generalise to $\RR^d$ what it means for a finite set $A$ to lie high above or deep below another finite set $B$.
For $k\leq d$, we write $\pi_{[k]}:\RR^d\to\RR^k$ for projection onto the first $k$ coordinates. 

\begin{defn}
Let $S,T\subset\RR^d$ be non-empty finite sets. We say that the pair $(S,T)$ is \emph{generic} if
\begin{enumerate}
\item $\pi_{[d-1]}$ is injective on $S$ and on $T$,
\item the affine subspaces spanned by $\pi_{[d-1]}(S)$ and $\pi_{[d-1]}(T)$ intersect at a unique point.
\end{enumerate}
\end{defn}

\begin{defn}
For a generic pair $(S,T)$, we say that \emph{$S$ lies above $T$} and that \emph{$T$ lies below $S$} if the unique pair of points $(s,t)$ with $s$ in the affine subspace spanned by $S$ and $t$ in the affine subspace spanned by $T$ satisfying $\pi_{[d-1]}(s)=\pi_{[d-1]}(t)$ also satisfies $\pi_d(s)>\pi_d(t)$. 
\end{defn}

\begin{defn} \label{def:dha}
For finite sets $A,B\subset\RR^d$, we say that \emph{$A$ lies high above $B$} and \emph{$B$ lies deep below $A$} if, for any generic pair $(S,T)$ with $S\subset A$ and $T\subset B$, $S$ lies above $T$. 
\end{defn}

Notice that for any two finite sets $A,B\subset \RR^d$, we can make $A$ lie high above $B$ by translating $A$ sufficiently high in the $d$-th coordinate. Indeed, for any generic pair $(S,T)$ with $S\subseteq A$ and $T\subseteq B$ (for which there are only finitely many choices), $S$ can be made to lie above $T$ by taking any sufficiently large translate of $S$ in the $d$-th direction.

For later use, we now wish to record a lemma that gives conditions under which we can deduce that a  point lies in the interior of the convex hull of a set. This will follow as a simple corollary of the next result.

\begin{lem}
\label{lem:above}
Let $S_1,S_2\subset\RR^d$ be non-empty sets such that $S_1$ lies deep below (or high above) $S_2$. Suppose there is a point $p\in S_1$ such that $\pi_{[d-1]}(p)\in\Int\conv(\pi_{[d-1]}(S_1\cup S_2))$. Then there is a point $q\in\conv(S_1\cup S_2)$ such that $\pi_{[d-1]}(q)=\pi_{[d-1]}(p)$ and $\pi_d(q)>\pi_d(p)$ (resp., $\pi_d(q)<\pi_d(p)$).
\end{lem}

\begin{proof}
Assume that $S_1$ lies deep below $S_2$. Pick any $r\in S_2$, so $r\neq p$. If $\pi_{[d-1]}(r)=\pi_{[d-1]}(p)$, then $(\set{r},\set{p})$ is a generic pair, so $\set{r}$ lies above $\set{p}$. This implies that $\pi_d(r)>\pi_d(p)$, so we can just pick $q=r$. Otherwise, let the open ray from $\pi_{[d-1]}(r)$ through $\pi_{[d-1]}(p)$ intersect the boundary of $\conv(\pi_{[d-1]}(S_1\cup S_2))$ at $x$. 
Such a point exists and is unique since $\pi_{[d-1]}(p)$ lies in the interior of $\conv(\pi_{[d-1]}(S_1\cup S_2))$. $x$ belongs to some face of the convex hull of $\pi_{[d-1]}(S_1\cup S_2)$, so it lies in $\conv(\pi_{[d-1]}(F))$ for some $F\subset S_1\cup S_2$ with $|F|=d-1$ and the affine span of $\pi_{[d-1]}(F)$ has dimension $d-2$. Furthermore, the affine span of $\pi_{[d-1]}(F)$ does not contain $\pi_{[d-1]}(p)$ or $\pi_{[d-1]}(r)$, so that $\pi_{[d-1]}(F\cup\{p\})$ and $\pi_{[d-1]}(F\cup\{r\})$ are in general position and each spans the entire space $\RR^{d-1}$ affinely. 
We also have that $\pi_{[d-1]}(p)\in\conv(\pi_{[d-1]}(F\cup\{r\}))$.

Set $S=(F\cup\{p\})\cap S_1$ and $T=(F\cup\{r\})\cap S_2$. We claim that $(S,T)$ is a generic pair. Note that $|S|+|T|=d+1$ and the dimension of the affine spans of $\pi_{[d-1]}(S)$ and $\pi_{[d-1]}(T)$ are $|S|-1$ and $|T|-1$, respectively. It will suffice to show that the affine spans of $\pi_{[d-1]}(S)$ and $\pi_{[d-1]}(T)$ intersect. Indeed, if the intersection is not unique, then it is an affine space of some dimension $d'\geq 1$, which implies that the dimension of the affine span of $\pi_{[d-1]}(S\cup T)=\pi_{[d-1]}(F\cup\{p,r\})$ is $|S|-1+|T|-1-d'\leq d-2$. But this would contradict the fact that $\pi_{[d-1]}(F\cup\{p,r\})$ affinely spans the entire space $\RR^{d-1}$.

Since $\pi_{[d-1]}(p)\in\conv(\pi_{[d-1]}(F\cup\{r\}))$, we can write
$$\pi_{[d-1]}(p)=\sum_{v\in F\cup\{r\}}c_v \pi_{[d-1]}(v)$$
for some $c_v\in [0,1]$ with $\sum_v c_v=1$. Since the affine span of $\pi_{[d-1]}(F)$ does not contain $\pi_{[d-1]}(p)$, we must have $c_r>0$. Then
$$\pi_{[d-1]}(p)-\sum_{v\in S-\{p\}}c_v \pi_{[d-1]}(v)=\sum_{v\in T}c_v \pi_{[d-1]}(v),$$
so that the affine spans of $\pi_{[d-1]}(S)$ and $\pi_{[d-1]}(T)$ intersect at the common point
$$\frac{1}{1-\sum_{v\in S-\{p\}}c_v}\paren{\pi_{[d-1]}(p)-\sum_{v\in S-\{p\}}c_v \pi_{[d-1]}(v)}=\frac{1}{\sum_{v\in T}c_v}\paren{\sum_{v\in T}c_v \pi_{[d-1]}(v)}.$$
Hence, $(S,T)$ is a generic pair, so that $T$ lies above $S$. We let $s$ and $t$ be the unique pair of points in the affine subspaces spanned by $S$ and $T$ satisfying $\pi_{[d-1]}(s)=\pi_{[d-1]}(t)$ and $\pi_d(t)>\pi_d(s)$. In other words,
$$s=\frac{1}{\sum_{v\in T}c_v}\paren{p-\sum_{v\in S-\{p\}}c_v v},\quad t=\frac{1}{\sum_{v\in T}c_v}\paren{\sum_{v\in T}c_v v}.$$

If $\sum_{v\in S-\{p\}}c_v=0$, then $p=s$ and we may take $q=t$ as our desired point. 
Otherwise, let 
$$u=\frac{\sum_{v\in S-\{p\}}c_v v}{\sum_{v\in S-\{p\}}c_v}.$$
Then $p=(\sum_{v\in T}c_v)s+(\sum_{v\in S-\{p\}}c_v)u$, so $p$ lies on the line segment $su$. 
Hence, by setting
$$q=\paren{\sum_{v\in T}c_v}t+\paren{\sum_{v\in S-\{p\}}c_v}u,$$
which is on the line segment $tu$, we get $\pi_{[d-1]}(q)=\pi_{[d-1]}(p)$ and $\pi_d(q)>\pi_d(p)$.
\end{proof}

\begin{lem}
\label{lem:uphull}
Let $S_0,S_1,S_2\subset \RR^d$ be non-empty sets such that $S_1$ lies deep below $S_2$ and $S_0$ lies deep below $S_1\cup S_2$. If there is a point $p\in S_1$ such that $\pi_{[d-1]}(p)\in \Int\conv (\pi_{[d-1]}(S_1\cup S_2))$, then $p\in \Int\conv (S_0\cup S_1\cup S_2)$.
\end{lem}

\begin{proof}
Applying Lemma~\ref{lem:above} to $S_1,S_2$, we get a point $q_2\in\conv(S_1\cup S_2)$ with $\pi_{[d-1]}(q_2)=\pi_{[d-1]}(p)$ and $\pi_d(q_2)>\pi_d(p)$. Applying Lemma~\ref{lem:above} again to $S_1\cup S_2,S_0$, we get a point $q_1\in\conv(S_0\cup S_1\cup S_2)$ with $\pi_{[d-1]}(q_1)=\pi_{[d-1]}(p)$ and $\pi_d(q_1)<\pi_d(p)$. It follows that $p\in \Int\conv (S_0\cup S_1\cup S_2)$.
\end{proof}

We now come to our definition of higher-dimensional Horton sets. For the remainder of this section, we fix a sequence $p_2,p_3,p_4,\ldots$ of prime numbers.

\begin{defn}
\label{def:horton}
For $k \leq d$, we say that a set $H\subset\RR^d$ is \emph{$(d,k)$-Horton with respect to the sequence $p_{k+1},\ldots,p_d$} (though we will just say $(d,k)$-Horton if the sequence is clear from context) if $H$ is a finite levelled set whose level map is of the form $\phi_H=a\pi_1+b$ for some $a,b\in\RR$ with $a\neq 0$, $\pi_{[k]}$ is injective on $H$ and, assuming these conditions hold throughout, $H$ is $(d,k)$-Horton according to a finite number of applications of the following rules:
\begin{enumerate}
\item If $H \subset \RR^k$, then $H$ is $(k,k)$-Horton.
\item If $|\pi_1(H)|\leq 1$, then $H$ is $(d,k)$-Horton.
\item If $d>k$ and $H$ satisfies the following:
\begin{enumerate}
\item $\pi_{[d-1]}(H)$ is $(d-1,k)$-Horton,
\item the sets $H_{i,p_d}$ are $(d,k)$-Horton for $i=0,1,\ldots,p_d-1$,
\item any index set $I\subset\set{0,1,\ldots,p_d-1}$ with $|I|\geq 2$ can be decomposed into a pair of non-empty sets $(J, I-J)$ such that the set $\bigcup_{i\in J} H_{i,p_d}$ lies deep below the set $\bigcup_{i\in I-J} H_{i,p_d}$,
\end{enumerate}
then the set $H$ is $(d,k)$-Horton.
\end{enumerate}
\end{defn}

Notice that if $H$ is $(d,k)$-Horton, then, for $\varepsilon$ sufficiently small, any $\varepsilon$-perturbation $H'$ of $H$ preserving the first coordinate is also $(d,k)$-Horton. 

\begin{rmk}
This definition does not generalise our previous definition of a Horton set in dimension 2. Indeed, Definition~\ref{def:2horton} required that $\phi_H(H)$ be a consecutive set of integers, but our definition of a $(d,k)$-Horton set does not. However, if a set $H$ is $(2,1)$-Horton with respect to the sequence $p_2=2$ and satisfies this property, then it is Horton as previously defined.
\end{rmk}

Though we draw heavily on Valtr's definition~\cite[Definition 5.1]{V92}, 
our definition of higher-dimensional Horton sets differs from his in several ways. Firstly, just like in the plane, it is defined in terms of levelled sets and so, in that way, is less general. However, our definition is more general in some other ways. For one thing, we do not require that the set be in strongly general position, which, in particular, means that there may be points with the same $x$-coordinate. This is essential for us, as we will be studying Horton sets that resemble sets of lattice points. Our definition also has a new parameter $k$, which allows us to build Horton sets starting from a $k$-dimensional set. The definition of a $d$-Horton set in \cite{V92} is instead akin to a $(d,1)$-Horton set.

We now generalise Lemma~\ref{lem:subhorton2} to higher dimensions. We will use a notion of \emph{complexity} for $(d, k)$-Horton sets $H$, where those $H$ with $|\pi_1(H)|\leq 1$ have complexity $0$ and $H$ has complexity $c$ if the sets $H_{i,p_d}$ in condition 3(b) each have complexity at most $c-1$ and at least one of them has complexity $c-1$.

\begin{lem} \label{lem:subhorton}
If $k\leq d$ and $H\subset \RR^d$ is $(d,k)$-Horton, then, for any integers $a$ and $p$ with $p\geq 1$, the set $H_{a,p}$ is also $(d,k)$-Horton.
\end{lem}

\begin{proof}
We will induct on $p$, $d$ and the complexity of the $(d,k)$-Horton set $H$. More precisely, we induct on the lexicographic ordering of $(d, p+c)$, where $c$ is the complexity. If $d=k$, then $H$ is 
a levelled set in $\RR^k$ with level map of the form $\phi_H=a\pi_1+b$ for some $a,b\in\RR$ with $a\neq 0$ 
by rule 1 of Definition~\ref{def:horton}, so $H_{a,p}$ is also $(d,k)$-Horton by the same rule. If $H$ has complexity 0, i.e., $|\pi_1(H)|\leq 1$, then the result is also clear. Thus, we may assume that $H$ has been built from sets of lower complexity as in rule 3.

If $p_d\mid p$, then  $H_{a,p}=(H_{a,p_d})_{a',p/p_d}$ for some integer $a'$, where $H_{a,p_d}$ is $(d,k)$-Horton by assumption. Since $H_{a, p_d}$ has lower complexity and $p/p_d<p$, our induction hypothesis implies that  $(H_{a,p_d})_{a',p/p_d}$ is $(d,k)$-Horton. So we only need to consider the case where $p_d$ and $p$ are coprime. To check that $H_{a,p}$ is $(d,k)$-Horton, we need to check that conditions 3(a), 3(b) and 3(c) of Definition~\ref{def:horton} are satisfied. 

For 3(a), since $\pi_{[d-1]}(H)$ is $(d-1,k)$-Horton, our induction hypothesis implies that $\pi_{[d-1]}(H_{a,p}) = \pi_{[d-1]}(H)_{a,p}$ is also $(d-1,k)$-Horton.

For 3(b), since $(H_{a,p})_{i,p_d}=(H_{i',p_d})_{a',p}$ 
for some integers $i'$ and $a'$ and $H_{i',p_d}$ has lower complexity, $(H_{a,p})_{i,p_d}$ is $(d,k)$-Horton by the induction hypothesis. 

For 3(c), suppose $I\subset\set{0,1,\ldots,p_d-1}$ with $|I|\geq 2$. Note that $(H_{a,p})_{i,p_d}=(H_{i',p_d})_{a_i,p}=H_{x,pp_d}$ for some $i', a_i, x$ such that $x=a+ip=i'+a_ip_d$ and $i'\in\{0,1,\ldots,p_d-1\}$, where $i'$ is uniquely determined by $i'\equiv a+ip\pmod{p_d}$. Since $p$ and $p_d$ are coprime, as $i$ ranges over $0,\ldots,p_d-1$, so does $i'$, in some permutation. Let $I'\subset\set{0,1,\ldots,p_d-1}$ be the image of $I$ under this permutation. 
Since $H$ is $(d,k)$-Horton, $I'$ can be decomposed into non-empty sets $(J',I'-J')$ such that $\bigcup_{i'\in J'} H_{i',p_d}$ lies deep below the set $\bigcup_{i'\in I'-J'} H_{i',p_d}$, so,  taking subsets, $\bigcup_{i'\in J'} (H_{i',p_d})_{a_i,p}$ lies deep below the set $\bigcup_{i'\in I'-J'} (H_{i',p_d})_{a_i,p}$. The partition $(J',I'-J')$ of $I'$ corresponds to a partition $(J,I-J)$ of $I$, so $\bigcup_{i\in J} (H_{a,p})_{i,p_d}$ lies deep below the set $\bigcup_{i\in I-J} (H_{a,p})_{i,p_d}$ and 3(c) is satisfied.
\end{proof}

We now show that any $(d-1, k)$-Horton set can be lifted to a $(d, k)$-Horton set. 

\begin{lem} \label{lem:horton}
Let $k< d$ and $B\subset \RR^{d-1}$ be $(d-1,k)$-Horton with level map of the form $a\pi_1+b$. 
Let $\overline{B}\subset\RR^d$ be the copy of $B$ with the last coordinate of all  points equal to 0. Then there are real numbers $a_x$ for each $x\in \pi_1(B)$ such that the set
$$H=\setcond{\vec{x}+a_{x_1}\vec{e}_d}{\vec{x}\in \overline{B}}$$
is $(d,k)$-Horton, where $x_1 = \pi_1(\vec{x})$.
\end{lem}

\begin{proof}
We shall induct on the size of the $(d-1,k)$-Horton set $B$ and the length of the interval containing $\phi_B(B)$. If $|\pi_1(B)|\leq 1$, there is nothing to prove. For any other $B$ and any $0 \leq i \leq p_d-1$, Lemma~\ref{lem:subhorton} implies that the set $B_i=B_{i,p_d}$ is $(d-1,k)$-Horton. Since either $|B_i|<|B|$ or $|B_i|=|B|$ and the interval containing $\phi_{B_i}(B_i)$ is shorter than the interval containing $\phi_B(B)$, our induction hypothesis implies that there are $(d,k)$-Horton sets $H_i$ of the required form corresponding to the $(d-1,k)$-Horton sets $B_i$. We shift the sets $H_i$ along the $d$-th axis in such a way that $H_i$ lies high above $\bigcup_{j=0}^{i-1} H_j$, where $H_1, \dots, H_{i-1}$ have already been shifted. We then take $H=\bigcup_{j=0}^{p_d-1} H_j$. This set is of the required form. To show that $H$ is $(d,k)$-Horton, we only need to check condition 3(c) of Definition~\ref{def:horton}. But, for any $I\subset\{0,1,\ldots,p_d-1\}$, if $z=\max I$, then we may set $J=I\setminus \{z\}$.
\end{proof}

A similar result holds for lifting many Horton sets simultaneously.

\begin{lem}
\label{lem:unif-horton}
Let $k\geq 0$ and $B_i\subset\RR^{d_i-1}$ be $(d_i-1,k)$-Horton with the same level map of the form $a\pi_1+b$ for $i=1,\ldots,n$, where $d_i>k$. Let $\overline{B}_i\subset\RR^{d_i}$ be the copy of $B_i$ with the last coordinate of all points equal to 0. Then there are real numbers $a_x$ for each $x\in \bigcup_i \pi_1(B_i)$ such that the sets
$$H_i=\setcond{\vec{x}+a_{x_1}\vec{e}_{d_i}}{\vec{x}\in \overline{B}_i}$$
are $(d_i,k)$-Horton for $i=1,\ldots,n$. Furthermore, the numbers $a_x$ can be chosen to be bounded above in absolute value by any positive real number.
\end{lem}

\begin{proof}
The proof is exactly the same as for Lemma~\ref{lem:horton}, except that we must be careful to shift the different $B_i$ in a consistent manner. For the final assertion, note that if the numbers $a_x$ work, then the rescaled numbers formed by multiplying by any non-zero number will also work, so we can make all the terms as small as we please. 
\end{proof}

The final result of this section says that, under appropriate conditions, if a certain projection of a subset of a Horton set is well-spread, then so is the subset itself.

\begin{lem}
\label{lem:goup}
Given $k\leq d$ and positive integers $N$ and $r$ such that $r$ is coprime to $p_{k+1}p_{k+2}\cdots p_d$, let $L\subset\RR^d$ be $(d,k)$-Horton and  $H\subset L$ be such that $\pi_{[k]}(H)$ is $(N,rp_{k+1}p_{k+2}\cdots p_d)$-well-spread in $\pi_{[k]}(L)$. Then $H$ is $(2^{d-k}N,r)$-well-spread in $L$.
\end{lem}

\begin{proof}
The case $k=d$ is clear. 
Suppose now that we have proved the lemma for all $d$ in the case where $k=d-1$. Then, for general $d>k$, since $\pi_{[k+1]}(L)$ is $(k+1,k)$-Horton, the case where $d=k+1$ implies that $\pi_{[k+1]}(H)$ is $(2N, rp_{k+2}\cdots p_d)$-well-spread in $\pi_{[k+1]}(L)$. Repeating this argument using the fact that an $(\ell, k)$-Horton set is also $(\ell, \ell-1)$-Horton for each $k+2 \leq \ell \leq d$, we see that $H$ is $(2^{d-k}N, r)$-well-spread in $L$. Thus, we only have to consider the case where $k=d-1$.

Suppose that the lemma is false. Let $H\subset L$ be the smallest pair, in the sense that $(|L|,|H|)$ is lexicographically smallest, that is a counterexample to the lemma. That is, $\pi_{[d-1]}(H)$ is $(N,rp_d)$-well-spread in $\pi_{[d-1]}(L)$, but $H$ is not $(2N,r)$-well-spread in $L$. In particular, $|L_{a,p_d}|\geq 1$ for all $a$, so that $|L_{a,p_d}|<|L|$ for all $a$. 

Note that for any integers $a$ and $p$ with $p \geq 1$, $\pi_{[d-1]}(H_{a,p})$ is  $(N,rp_d)$-well-spread 
in $\pi_{[d-1]}(L_{a,p})$ and, by Lemma~\ref{lem:subhorton}, $L_{a,p}$ is $(d,k)$-Horton, so we only have to show that $H$ is $(2N,r)$-spread in $L$. 

Suppose $S\subset H$ is any subset of size $2N$ and $a'$ is any integer. Then we wish to show that $\Int\conv S$ contains a point of $L_{a',r}$. Let $I=\setcond{i\in [p_d]}{S_{i,p_d}\neq\emptyset}$. If $|I|=1$, say $I=\{i\}$, then, since $\pi_{[d-1]}(H_{i,p_d})$ is  $(N,rp_d)$-well-spread in $\pi_{[d-1]}(L_{i,p_d})$, our minimality assumption implies that $H_{i,p_d}$ is $(2N,r)$-well-spread in $L_{i,p_d}$. Hence, since $p_d$ and $r$ are coprime, there exists an integer $z$ depending on $a'$ such that $\Int\conv S$ contains a point of $(L_{i,p_d})_{z,r}\subset L_{a',r}$. 

So assume that $|I|>1$. Then we can partition $I$ into non-empty $J$ and $I-J$ so that $L_{lo}=\bigcup_{i\in J} L_{i,p_d}$ lies deep below the set $L_{hi}=\bigcup_{i\in I-J} L_{i,p_d}$. Set $S_{lo}=S\cap L_{lo}$ and $S_{hi}=S\cap L_{hi}$ and suppose, without loss of generality, that $|S_{lo}|\leq |S_{hi}|$, so $|S_{hi}|\geq N$. View $S_{hi}$ as a levelled set with the same level map as $L$. Let $I_m=\setcond{i\in [p_d^m]}{(S_{hi})_{i,p_d^m} \neq \emptyset}$ and let $m$ be the smallest positive integer such that $|I_m|>1$. Such an $m$ exists, since $\pi_{[d-1]}(H)$ is $(N,rp_d)$-spread in $\pi_{[d-1]}(L)$ and $|\pi_{[d-1]}(S_{hi})|=|S_{hi}|\geq N$, so $\pi_{[d-1]}(S_{hi})$ contains at least one point of $\pi_{[d-1]}(L)$ in the interior of its convex hull, which in turn implies that $S_{hi}$ cannot lie on a hyperplane perpendicular to the first axis, i.e., $|\phi_L(S_{hi})|>1$. 

Let $i_0$ be the unique element of $I_{m-1}$. Since $L$ is $(d,k)$-Horton, so is $L_{i_0,p_d^{m-1}}$ by Lemma~\ref{lem:subhorton}. 
Now partition $I_m$ into non-empty $J'$ and $I_m-J'$ so that $L_{hilo}=\bigcup_{i\in J'} L_{i,p_d^m}$ lies deep below the set $L_{hihi}=\bigcup_{i\in I_m-J'} L_{i,p_d^m}$. Fix any $j\in J'$. Then, since $j\equiv i_0\pmod{p_d^{m-1}}$, we have $L_{j,p_d^m}\cap L_{a',r}=L_{z,rp_d^m}=(L_{i_0,p_d^{m-1}})_{z',rp_d}$ for some integers $z$ and $z'$. But $\pi_{[d-1]}(H_{i_0,p_d^{m-1}})$ is $(N,rp_d)$-spread in $\pi_{[d-1]}(L_{i_0,p_d^{m-1}})$ and, hence, since $|\pi_{[d-1]}((S_{hi})_{i_0,p_d^{m-1}})| = |\pi_{[d-1]}(S_{hi})| \geq N$, 
$\Int\conv \pi_{[d-1]}((S_{hi})_{i_0,p_d^{m-1}})$ contains a point of $\pi_{[d-1]}((L_{i_0,p_d^{m-1}})_{z',rp_d}) = \pi_{[d-1]}(L_{z,rp_d^m})$, 
say $\pi_{[d-1]}(\vec{x})$ with $\vec{x}\in L_{z,rp_d^m}$. Since $L_{z,rp_d^m}\subseteq L_{j,p_d^m}\subseteq L_{hilo}$, we have $\vec{x}\in L_{hilo}$. Applying Lemma~\ref{lem:uphull} with $S_0=S_{lo}$, $S_1=(S_{hi}\cap L_{hilo})\cup\{\vec{x}\}$ and $S_2=S_{hi}\cap L_{hihi}$, noting that $S_2\subseteq L_{hihi}$ is high above $S_1\subseteq L_{hilo}$ and $S_0\subseteq L_{lo}$ is deep below $S_1\cup S_2\subseteq L_{hi}$, we get that $\vec{x}\in\Int\conv S$ with $\vec{x}\in L_{z,rp_d^m}\subset L_{a',r}$.
\end{proof}

Note that if $L$ is a $(d, 1)$-Horton set with level map $\pi_1$ and $\pi_1(L)$ is a consecutive set of integers, then $\pi_1(L)$ is $(N, p_2 \cdots p_d)$-well-spread in itself provided $N \geq p_2 \cdots p_d + 2$. But then Lemma~\ref{lem:goup} implies that $L$ is $(2^{d-1} N, 1)$-spread in itself, which is the same as saying that $L$ is $2^{d-1} N$-hole-free. This is the essence of Valtr's construction of high-dimensional hole-free sets.

\subsection{The construction}

We now come to the proof of our main technical result, Theorem~\ref{thm:lattice}, which, we recall, is a construction of point sets in $\RR^d$ which lie arbitrarily close to the lattice cube $[n]^d$ and are $C'_d$-hole-free for some constant $C'_d$ depending only on $d$.

Fix $n>1$, let $p_1<p_2<\cdots<p_d$ be the first $d$ primes and recall that  $\vec{e}_1,\ldots,\vec{e}_d$ is the standard basis for $\RR^d$. Let $I=[n]^d$, which we will use as an indexing set. We start with the lattice cube $\cP^{(1,0)}=[n]^d=\setcond{P^{(1,0)}_{\vec{x}}}{\vec{x}\in I}$, where $P^{(1,0)}_{\vec{x}}=\vec{x}$. For any such set $\cP=\setcond{P_{\vec{x}}}{\vec{x}\in I}$ indexed by $I$ and any set $J\subset \RR^d$, we denote by $\cP_J$ the set
$$\setcond{P_{\vec{x}}}{\vec{x}\in J\cap I}.$$
We shall perturb $\cP^{(1,0)}$ in several steps, to obtain a sequence of sets and negligible bijections 
$$\cP^{(1,0)}\to \cP^{(1,1)}\to\cdots\to \cP^{(1,d)}=\cP^{(2,0)}\to\cdots\to \cP^{(2,d)}=\cP^{(3,0)}\to\cdots\to \cP^{(d,d)}.$$

For $i>1$, we will set $\cP^{(i,0)}=\cP^{(i-1,d)}$. Given the set $\cP^{(i,j-1)}=\setcond{P^{(i,j-1)}_{\vec{x}}}{\vec{x}\in I}$ for some $1\leq j\leq d$, we construct $\cP^{(i,j)}=\setcond{P^{(i,j)}_{\vec{x}}}{\vec{x}\in I}$ as follows. If $j=i$, then there is no change, i.e.,  $P^{(i,j)}_{\vec{x}}=P^{(i,j-1)}_{\vec{x}}$. For $j\neq i$, we perturb the planes perpendicular to $\vec{e}_i$ in the $j$-th direction, setting 
$$P^{(i,j)}_{\vec{x}}=P^{(i,j-1)}_{\vec{x}}+a_{x_i}\vec{e}_j$$
for $\vec{x}=(x_1,\ldots,x_d)\in I$, 
where $a_1,\ldots,a_n$ will be some carefully chosen constants. Before deciding their values, we make some definitions.

Let $V$ be any affine subspace of dimension $k$ such that 
$V$ is perpendicular to $\vec{e}_1,\ldots,\vec{e}_{i-1}$  but not to $\vec{e}_i$ and $V\cap\ZZ^d$ is a $k$-dimensional lattice.
Let $1\leq j_{k+1}<j_{k+2}<\cdots<j_d$ be such that, for $k+1\leq l\leq d$, $V+\RR\vec{e}_1+\cdots +\RR\vec{e}_{j_l-1}$ is of dimension $l-1$ but $V_l=V+\RR\vec{e}_1+\cdots +\RR\vec{e}_{j_l}$ is of dimension $l$, so we have a flag of affine subspaces $V=V_k\subset V_{k+1}\subset\cdots\subset V_d=\RR^d$. We define an affine isomorphism $\varphi_V:\RR^d\to \RR^d$ in steps using the flag $V_k\subset\cdots\subset V_d$ by defining $\varphi_V:V_l\to\RR^l\subset\RR^d$ (the embedding $\RR^l\subset\RR^d$ is given by taking all but the first $l$ coordinates to be 0). Note that $V\cap I$ is a levelled set with level map $\pi_i$. We begin by defining $\varphi_V:V\to\RR^k$ to be any affine isomorphism which sends this level map to $\pi_1$ (in other words, $\pi_1\circ \varphi_V=\pi_i$ on $V$). To define $\varphi_V:V_l\to \RR^l$ for $l > k$, note that any vector in $V_l$ can be written uniquely as $\vec{x}+r\vec{e}_{j_l}\in V_l$ with $\vec{x}\in V_{l-1}$ and $r\in\RR$, so we may set $\varphi_V(\vec{x}+r\vec{e}_{j_l})=\varphi_V(\vec{x})+r\vec{e}_l$. 

Now $a_1,\ldots,a_n$ are chosen such that the following properties hold:

\begin{enumerate}
\item The natural bijection $\Phi:\cP^{(i,j-1)}\to \cP^{(i,j)}$ is negligible.
\item For any affine subspace $W$ of dimension $k$ which is parallel to $\vec{e}_j$ and any set $S\subset \cP^{(i,j-1)}\cap W$, the bijection $S\to \Phi(S)$ is negligible in the ambient space $W\cong\RR^k$.
\item For any $V,k,\varphi_V$ as above, the bijection  $\pi_{[k]}(\varphi_V(\cP^{(i,j-1)}_V))\to \pi_{[k]}(\varphi_V(\cP^{(i,j)}_V))$ is negligible.
\item For any $V,k,\varphi_V, j_{k+1},\ldots,j_d$ as above and any 
$m\geq k+1$, there is some $\vec{v}\in\RR^d$, depending only on $V$, such that the set $\varphi_V(\cP^{(i,j_m)}_{V}-\vec{v})\subset \RR^m$ is $(m,k)$-Horton with respect to the sequence $p_{j_{k+1}},\ldots,p_{j_{m}}$.
\end{enumerate}

Properties 1, 2 and 3 will be satisfied for any sufficiently small sequence $a_1,\ldots,a_n$. (To see that Property 2 holds, we only need to consider the finitely many possible subsets $S$ of $\cP^{(i,j-1)}$.) 
To show that property 4 may also be satisfied, note that $\cP^{(i,0)}_{V}$ is a translate of $V\cap I$. Indeed, since $V$ is perpendicular to $\vec{e}_l$ for $l=1,\ldots,i-1$, each of the perturbations in the sequence $\cP^{(l,0)}\to \cP^{(l,1)}\to \cdots\to \cP^{(l+1,0)}$  translates each point of $\cP^{(l,k)}_V$ by the same amount. Thus, $\cP^{(i,0)}_{V}$ is a translate of $V\cap I$, 
say $\cP^{(i,0)}_{V}=(V\cap I)+\vec{v}$. So $\varphi_V(\cP^{(i,0)}_{V}-\vec{v})=\varphi_V(V\cap I)\subset \RR^{k}$ is $(k,k)$-Horton. Suppose now that we have shown inductively that $\cP^{(i,j_{m-1})}_{V}-\vec{v}\subset V_{m-1}$ and that $\varphi_V(\cP^{(i,j_{m-1})}_{V}-\vec{v})\subset \RR^{m-1}$ is $(m-1,k)$-Horton with respect to the sequence $p_{j_{k+1}},\ldots,p_{j_{m-1}}$. From property 2 above, $\cP^{(i,j_{m-1})}_{V}\to \cP^{(i,j_m-1)}_V$ is negligible in the ambient space $V_{m-1}+\vec{v}$, so that $\varphi_V(\cP^{(i,j_m-1)}_{V}-\vec{v})\subset \RR^{m-1}$ 
is also $(m-1,k)$-Horton. $\cP^{(i,j_m)}_{V}-\vec{v}\subset V_{m}$ will always hold, no matter the choice of $a_1,\ldots,a_n$, so we only need to choose $a_1,\ldots,a_n$ so that $\varphi_V(\cP^{(i,j_m)}_{V}-\vec{v})\subset \RR^{m}$ is $(m,k)$-Horton with respect to the sequence $p_{j_{k+1}},\ldots,p_{j_{m}}$. But Lemma~\ref{lem:unif-horton} guarantees that such a choice exists and that $a_1,\ldots,a_n$ can be taken to be arbitrarily small.

Thus, we have constructed a collection of sets and bijections 
$$\cP^{(1,0)}\to \cP^{(1,1)}\to\cdots\to \cP^{(1,d)}=\cP^{(2,0)}\to\cdots\to \cP^{(2,d)}=\cP^{(3,0)}\to\cdots\to \cP^{(d,d)},$$
where, for any affine subspace $V$ of dimension $k$ such that $V\cap\ZZ^d$ is a $k$-dimensional lattice, there are $i$ and $j$ such that $\cP^{(i,j)}_{V}$ is affine isomorphic to a $(d,k)$-Horton set with respect to some subsequence of $p_1,\ldots,p_d$. Moreover, since all the bijections are negligible by property 1 and the composition of negligible bijections is itself negligible, the bijection $\cP^{(i,j)}\to \cP^{(d,d)}$ is also negligible.

We remark that just as $[n]^d$ is the Minkowski sum of $[n]$ along each of the $d$ axes, $\cP^{(d,d)}$ is also the Minkowski sum of some perturbation of $[n]$ along each of the $d$ axes.

\begin{thm}
\label{thm:hdd}
The set $\cP^{(d,d)}$ constructed above is $C_d'$-hole-free for some constant $C_d'$.
\end{thm}

Our proof goes as follows. We inductively find constants $N_1,N_2,\ldots,N_d$ such that, for each $k$, for any $k$-dimensional affine subspace $V_k\subset\RR^d$ and any $S \subset I_k = I\cap V_k$ of size $N_k$, the set $\cP^{(d,d)}_S\subset \cP^{(d,d)}$ contains some element of $\cP^{(d,d)}_{I_k}$ in the interior of its convex hull. Then we may set $C_d'$ to be $N_d$. 
At each step, Lemma~\ref{lem:well-spread} tells us that either there are $N_{k-1}$ points of $S$ on a hyperplane $V_{k-1}$, in which case we are done by induction, or $S$ is well-spread in $I_k$. In the latter case, using the fact that, by property 4 of our construction, the set $\cP^{(d,d)}_{I_k}$ is (up to a negligible perturbation) affine isomorphic to a $(d, k)$-Horton set, we can apply Lemma~\ref{lem:goup} to conclude that there is a point of $\cP^{(d,d)}_{I_k}$ in the interior of the convex hull of $\cP^{(d,d)}_S$.

\begin{proof}[Proof of Theorem~\ref{thm:hdd}]
Let $V_1\subset\RR^d$ be any 1-dimensional affine subspace, so that $I_1=I\cap V_1$ is a consecutive subset of (i.e., the intersection of a line segment with) a 1-dimensional lattice $L$. We first show that $\cP^{(d,d)}_{I_1}$ is $N_1$-hole-free for some $N_1$. Let $i\geq 1$ be the smallest integer such that $L$ is not perpendicular to $\vec{e}_i$ and $j\geq 1$ the smallest integer such that $V_1+\RR\vec{e}_1+\cdots+\RR\vec{e}_j$ is of dimension $d$. 
By property 4 of the construction, $\cP^{(i,j)}_{I_1}$ is affine isomorphic to a $(d,1)$-Horton set $H$ with respect to some subsequence $q_2,\ldots,q_d$ of $p_1,p_2,\ldots,p_d$. 
Note that $\pi_1(H)$ is a consecutive subset of a sublattice of $\ZZ$, so, by modifying the level map of $H$ if necessary, we may assume that $\phi_H(H)$ is a set of consecutive integers. Since any consecutive subset of the 1-dimensional lattice $\ZZ\subset\RR$ is $(m+2,m)$-well-spread in itself for any $m$, $\pi_{[1]}(H)$ is $(q_2\cdots q_d+2,q_2\cdots q_d)$-well-spread in itself, so, by Lemma~\ref{lem:goup}, $H$ is $(2^{d-1}(q_2\cdots q_d+2),1)$-well-spread in itself. Therefore, setting $N_1=2^{d-1}(p_2\cdots p_d+2)$, we have that the interior of the convex hull of any $N_1$ points of $\cP^{(i,j)}_{I_1}$ contains a point of $\cP^{(i,j)}_{I_1}$. Since $\cP^{(i,j)}_{I_1}\to \cP^{(d,d)}_{I_1}$ is negligible, $\cP^{(d,d)}_{I_1}$ is $N_1$-hole-free.

Suppose we have found $N_1,\ldots,N_{k-1}$ such that, for $j=1,\ldots,k-1$, any $j$-dimensional affine subspace $V_j\subset\RR^d$ for which $I_j=I\cap V_j$ is part of a $j$-dimensional lattice has the property that the set $\cP^{(d,d)}_{I_j}$ is $N_j$-hole-free. By Lemma~\ref{lem:well-spread}, if $N = N_{k-1} (p_{k+1}\cdots p_d+1)^k k^{2k}$, then any subset of $\ZZ^k$ either has at least $N_{k-1}$ points on a hyperplane of $\RR^k$ or is $(N,q_{k+1}\cdots q_d)$-well-spread in $\ZZ^k$ for any subsequence $q_{k+1},\ldots, q_d$ of $p_1,\ldots,p_d$. Consider a $k$-dimensional affine subspace $V_k\subset\RR^d$ such that $I_k=I\cap V_k$ is part of a $k$-dimensional lattice. Let $i\geq 1$ be the smallest integer such that $V_k$ is not perpendicular to $\vec{e}_i$ and $j\geq 1$ the smallest integer such that $V_k+\RR\vec{e}_1+\cdots+\RR\vec{e}_j$ is of dimension $d$. By property 4 of the construction, there is some $\vec{v}\in\RR^d$ such that $H=\varphi_{V_k}(\cP^{(i,j)}_{I_k}-\vec{v})$ is $(d,k)$-Horton with respect to some subsequence $q_{k+1},\ldots,q_d$ of $p_1,\ldots,p_d$. 

We will find some $N_k>N_{k-1}$ such that $\cP^{(d,d)}_{I_k}$ is $N_k$-hole-free. View $V_k\cap \ZZ^d$ as a levelled set with level map of the form $\phi=a\pi_i+b$ such that $\phi(V_k\cap \ZZ^d)=\ZZ$. Note that $\phi$ may be seen as the projection onto a basis element of some $\ZZ$-basis of $V_k\cap \ZZ^d$. Suppose $S\subset I_k$. If $S$ contains $N_{k-1}$ points on a hyperplane $V_{k-1}\subset V_k$, then, for $I_{k-1} = I \cap V_{k-1}$, the interior of the convex hull of $\cP^{(d,d)}_{S}$ contains a point of $\cP^{(d,d)}_{I_{k-1}}$ by our induction hypothesis. Otherwise, by the application of Lemma~\ref{lem:well-spread} discussed above, $S$ is $(N,q_{k+1}\cdots q_d)$-well-spread in $V_k\cap \ZZ^d$ (in the ambient space $V_k\cong \RR^k$) 
for any subsequence $q_{k+1},\ldots, q_d$ of $p_1,\ldots,p_d$. By intersecting with the convex hull of $I_k$, $S$ is $(N,q_{k+1}\cdots q_d)$-well-spread in $I_k$. Applying $\varphi_{V_k}$, we have that  $\varphi_{V_k}(S)=\pi_{[k]}(\varphi_{V_k}(\cP^{(i,0)}_S-\vec{v}))$ is $(N,q_{k+1}\cdots q_d)$-well-spread in $\varphi_{V_k}(I_k)=\pi_{[k]}(\varphi_{V_k}(\cP^{(i,0)}_{I_k}-\vec{v}))$. Let $K=\varphi_{V_k}(\cP^{(i,j)}_S-\vec{v})\subset H$. 
By property 3 of the construction, $\pi_{[k]}(\varphi_{V_k}(\cP^{(i,0)}_{I_k}))\to \pi_{[k]}(\varphi_{V_k}(\cP^{(i,j)}_{I_k}))$ is negligible. Therefore, since negligible maps preserve well-spreadness, 
$\pi_{[k]}(K)$ is $(N,q_{k+1}\cdots q_d)$-well-spread in $\pi_{[k]}(H)$. 
Hence, by Lemma~\ref{lem:goup}, $K$ is $(2^{d-k}N,1)$-well-spread in $H$. Setting $N_k=2^{d-k}N$, we see that any subset of $K$ of size $N_k$ contains a point of $H$ in the interior of its convex hull and, hence, that $H$ is $N_k$-hole-free. This implies that $\cP^{(i,j)}_{I_k}$ and, hence, $\cP^{(d,d)}_{I_k}$ is also $N_k$-hole-free.

Following the induction to $k = d$, we see that $\cP^{(d,d)}_{I}=\cP^{(d,d)}$ is $N_d$-hole-free, so taking $C_d'=N_d$ suffices to complete the proof.
\end{proof}

To complete the proof of Theorem~\ref{thm:lattice}, we note that $N_1=2^{d-1}(p_2\cdots p_d+2) = d^{d + o(d)}$ and $N_k = 2^{d-k} N_{k-1} (p_{k+1}\cdots p_d+1)^k k^{2k} \leq N_{k-1} d^{kd + o(kd)}$. Therefore, 
\[C'_d = N_d \leq \prod_{k=1}^d d^{kd + o(kd)} \leq d^{d^3/2 + o(d^3)},\]
as required. A slightly more careful analysis improves this to $d^{d^3/6 + o(d^3)}$, but we suspect that even this is quite far from the true bound.

\section{Concluding remarks}

If $\cP$ is a finite set of points in $\mathbb{R}^d$, define the \emph{spread} $q_d(\cP)$ of $\cP$ to be the  maximum distance between any two points of $\cP$ divided by the minimum distance between any two points of $\cP$. A simple volume argument shows that if $\cP$ has $k$ points, then $q_d(\cP) \geq \gamma_d k^{1/d}$ for some $\gamma_d > 0$ depending only on $d$, while the lattice $[n]^d$ shows that this bound is tight up to the constant. Answering a question of Alon, Katchalski and Pulleyblank~\cite{AKP89} inspired by the fact that Horton's original construction of $7$-hole-free sets has very large spread, Valtr~\cite{V92-2} showed that there is a constant $\alpha$ and, for every natural number $k$, a set $\cP$ of $k$ points in the plane such that $q_2(\cP) \leq \alpha \sqrt{k}$ but $\cP$ is still $7$-hole-free. In fact, this is a simple corollary of his result showing that there are $7$-hole-free sets which are $\varepsilon$-perturbations of the set of lattice points $[n]^2$. Similarly, our Theorem~\ref{thm:lattice} easily implies the following result.

\begin{cor}
For any integer $d \geq 2$, there exist  constants $\alpha_d$ and $C'_d = d^{O(d^3)}$ such that, for every natural number $k$, there is a set $\cP$ of $k$ points in $\mathbb{R}^d$ such that $q_d(\cP) \leq \alpha_d k^{1/d}$ and $\cP$ is $C'_d$-hole-free. In particular, when $d = 2$, one may take $C_2' = 7$.
\end{cor}

In~\cite{V92-2}, Valtr also 
addressed another question first studied by Alon, Katchalski and Pulleyblank~\cite{AKP89}, namely, given a set of $k$ points $\cP$ in general position in $\RR^2$ with $q_2(\cP) \leq \alpha \sqrt{k}$ for some constant $\alpha$, how large of a convex subset must $\cP$ contain? If we write $c_\alpha(k)$  for the size of the largest such subset, Valtr proved that $c_{\alpha}(k) \geq \beta k^{1/3}$ for some $\beta > 0$ depending only on $\alpha$ and also that this is best possible up to the constant. His construction showing that this is best possible is again just his construction of an $\varepsilon$-perturbation of the set of lattice points $[n]^2$ with no large holes, though the analysis requires significant additional work. It would be interesting to decide if our construction also impinges on the analogous problem in higher dimensions.

\vspace{3mm}
\noindent
{\bf Acknowledgements.} We are indebted to the anonymous reviewer for several insightful remarks.

\end{document}